\documentclass[a4paper, 11pt]{article}
\usepackage{amsmath,amssymb,esint,amscd,xspace,fancyhdr,color,authblk,srcltx,fontenc,bbm}
\usepackage{hyperref}
\usepackage{cite}

\newtheorem{theorem}{Theorem}[section]

\newtheorem{definition}[theorem]{Definition}

\newtheorem{lemma}[theorem]{Lemma}

\newtheorem{remark}[theorem]{Remark}
\newenvironment{proof}[1][Proof]{\textbf{#1.} }{\hfill\rule{0.5em}{0.5em}}
{\catcode`\@=11\global\let\AddToReset=\@addtoreset
\AddToReset{equation}{section}

\AddToReset{theorem}{section}

\title{Existence of weak solutions to borderline double-phase problems with logarithmic convection terms}
\author{Minh-Phuong Tran\footnote{Corresponding author.} \thanks{Applied Analysis Research Group, Faculty of Mathematics and Statistics, Ton Duc Thang University, Ho Chi Minh City, Vietnam; \texttt{tranminhphuong@tdtu.edu.vn}}, Thanh-Nhan Nguyen\thanks{Group of Analysis and Applied Mathematics, Department of Mathematics, Ho Chi Minh City University of Education, Ho Chi Minh City, Vietnam; \texttt{nhannt@hcmue.edu.vn}}}

\date{\today}

\begin{document}
 
\maketitle
\begin{abstract}

In this study, we devote our attention to the question of clarifying the existence of a weak solution to a class of quasilinear double-phase elliptic equations with logarithmic convection terms under some appropriate assumptions on data. The proof is based on the surjectivity theorem for the pseudo-monotone operators and modular function spaces and embedding theorems in generalized Orlicz spaces. Our approach in this paper can be extended naturally to a larger class of unbalanced double-phase problems with logarithmic perturbation and gradient dependence on the right-hand sides. \\



\end{abstract}   
                  

\section{Introduction}\label{sec:intro}

In this paper, we are concerned with the existence result for the following double-phase problem of the type
\begin{align}\label{main-eq}
\begin{cases} - \mathrm{div}\left(\mathcal{A}(\nabla u)\right) &= \ F(x,u,\nabla u), \quad  \mbox{ in } \Omega, \\ \hspace{1.3cm} u & = \ 0, \hspace{2cm} \mbox{ on } \partial \Omega, \end{cases}
\end{align}
where $\Omega \subset \mathbb{R}^n$ is an open bounded domain ($n \ge 2$) with smooth boundary $\partial \Omega$, and in this case, the vector field $\mathcal{A}: \mathbb{R}^n \to \mathbb{R}^n$ is given by
\begin{align}\label{def:A}
\mathcal{A}(\xi) = \partial_{\xi} \big(|\xi|^{p} + |\xi|^{p}\log(e+|\xi|)\big), \quad \xi \in \mathbb{R}^n.
\end{align}
Here, we restrict our study to the most interesting case of $1<p<n$ and the right-hand side of~\eqref{main-eq} is a gradient-dependent perturbation (convection term). We assume that $F: \Omega \times \mathbb{R} \times \mathbb{R}^n \to \mathbb{R}$ is a Carath\'eodory function (that is, it is measurable for a.e. $x \in \Omega$ with respect to $(t,\xi) \in \mathbb{R}\times\mathbb{R}^n$ and continuous for every $(t,\xi) \in\mathbb{R}\times\mathbb{R}^n$ with respect to $x \in \Omega$). 

The equation appearing in~\eqref{main-eq} is relevant to the Euler-Lagrange equation of the energy functional
\begin{align}\label{eq:g}
\omega \mapsto \int_\Omega{g(\nabla \omega)dx},
\end{align}
where the integrand $g$ is a combination of two different phases of elliptic behaviors: standard polynomial growth and the logarithmic perturbation of functional with $(p,q)$-growth with respect to the gradient. The study of functional displayed in~\eqref{eq:g} is included in the realm of integral functional with non-standard growth condition, according to Marcellini's pioneering contributions in~\cite{Marcellini1989, Marcellini1991}, where the energy density satisfies
\begin{align}\label{eq:pq}
|\xi|^p \lesssim g(\xi) \lesssim |\xi|^q+1, \qquad 1 < p \le q.
\end{align}

The theory of double-phase functionals traces a long way back to the works of Zhikov in~\cite{Zhikov1986, Zhikov2011}, who first described a feature of strongly anisotropic materials with two hardening exponents $p$ and $q$ in the context of homogenization theory and nonlinear elasticity. In such models, integrand $g$ changes its ellipticity rate according to the point, i.e. $g \equiv g(x,\xi)$. In particular, $g(x,\xi) = |\xi|^p + a(x)|\xi|^q$, where the coefficient $a(\cdot)$ dictates the geometry of a composite of two different materials. For this model, there have been extensive mathematical investigations pertaining to the existence of minimizers. In a related context, various existence results have been studied by multiple authors in different directions. To mention a few, an existence result proved by Perera and Squassina~\cite{PS18} using Morse theory (critical groups); existence and multiplicity results established by Liu and Dai~\cite{LD18} with variational methods; existence proofs by Gasi\'nski and Winkert~\cite{GW19} with the usage of surjectivity for pseudomonotone operators; the isotropic and anisotropic double-phase problems studied by  R\v{a}dulescu in~\cite{R2019}; Zhang and R\v{a}dulescu~\cite{ZR2018} with the tools of critical points theory in generalized Orlicz–Sobolev space; and an increasing list of references is far from being complete. Besides, many notable works in the regularity of minimizers to such models have been developed in recent years. To explore more, we refer to the studies of Colombo and Mingione in~\cite{CM2015,CM2015_2}, Baroni, Colombo, Mingione~\cite{BCM16,BCM18}, Byun and Oh in~\cite{BO2017}, Beck-Mingione~\cite{BM2020}, De Filippis-Mingione~\cite{FM2019, FM2020,FM2020_2} and the excellent survey paper~\cite{MR2021} for more references. 

To our knowledge, the energy functional of type~\eqref{eq:g} mixed with both polynomial and logarithmic perturbations was of interest for the first time in~\cite{MS1999, BCM15, BCM16} towards the regularity of minimizers. In these papers, the authors mentioned the borderline case of double-polynomial-phase functionals
\begin{align}\label{eq:G}
\omega \mapsto \mathcal{G}(\omega,\Omega):=\int_\Omega{\big(a(x)|\nabla\omega|^p + |\nabla\omega|^p\mathrm{log}(e+|\nabla\omega|) \big)dx},
\end{align}
where the non-negative function $a(\cdot)$ is supposed to be bounded. So far as we know, there is not much previous work regarding the existence and/or uniqueness of solutions to corresponding equations related to such energy functionals, especially when the right-hand side is a convection term (with gradient dependence). Our goal here is to address the question of the existence of at least a weak solution to the corresponding Euler-Lagrange equation of functional $\mathcal{G}$ shown in~\eqref{eq:G}. Inspired by recent works of existence theory for quasilinear double-phase problems with $(p,q)$-Laplacian and convection term mentioned above, we believe that it would be very interesting to treat the borderline-case of double-phase problem~\eqref{main-eq}. Furthermore, for simplicity purposes,  we focus on the incorporated functional $\mathcal{G}$ in~\eqref{eq:G}, where $a(\cdot) \equiv 1$, or the double-phase integrand $g$ depends only on the gradient of the solution. It should be mentioned that the case when the dependence on $x$ produces the gap, that is $\mathcal{A}(x,\xi)=\partial_{\xi} \big( a(x)|\xi|^{p} +|\xi|^{p}\mathrm{log}(e+|\xi|)\big)$, is also important and this will be investigated in a forthcoming paper. 

The main existence result in this paper is based on the subjectivity theorem for pseudo-monotone operators in~\cite{CLM07}. To the best of our knowledge, the theory of pseudo-monotone operators is known as a useful tool in proving the existence theorems for quasilinear elliptic and parabolic partial differential equations. The key idea underlying this existence theory, which goes back to Browder~\cite{Browder} and Minty~\cite{Minty}, is to study the properties of monotone operators. Later, the existence theorems were extended to a more general class of quasilinear equations by Hartmann and Stampacchia in~\cite{HS1966}. Since there were several equations involving operators not monotone, Br\'ezis in~\cite{Brezis} introduced a vast class of so-called pseudo-monotone operators, thereby extending Browder's existence theorem. This class of operators became an important research topic that has been developed in the solvability of a wide class of quasilinear equations. Therefore, the proof of such existence exploits an argument similar to that in~\cite{CLM07} and somewhat enhances previous results in the same direction~\cite{LD18, GW19}. The main novelty of this work lies in the model that, our problem is driven by a nonhomogeneous differential operator mixed between standard polynomial and logarithmic growths. Furthermore, it is worth mentioning that the reaction term is allowed to depend on both the solution and its gradient (called the convection term). Further, it is worth mentioning that when comparing with what has been done with quasilinear elliptic equations/systems driven by double-phase operators involving two distinct power-hardening exponents $p$ and $q$ according to the coefficient $a(\cdot)$, we observe that the previous existence results cannot imply the existence result in this paper. These features make our main result somewhat new and interesting. 

Now we spend a few words about our technique. Although the proof is based on the surjectivity theorem for the pseudo-monotone operators (see~\cite{CLM07}) and the key idea follows a similar path to the proof for double-phase problems with $(p,q)$-growth in~\cite{GW19}, due to the logarithmic growth, the connection between embedding theorems and the use of modular functions make a difference in our approach here. So, we believe that such a technique is perhaps interesting in itself and could find further possible applications in many different special models with the appearance of logarithmic growth.

This paper is organized as follows. Section~\ref{sec:pre} is devoted to notation, definitions, and necessary background from the theory of double-phase problems and function spaces. In this section, we also sketch the main ideas, and ingredients employed in the proofs of our result and point out some technical difficulties. Section~\ref{sec:main} is the most substantial, it contains the statement of the main theorem, and some remarks are concerned with both the model and technique. Finally, in Section~\ref{sec:proof}, we end up proving the main result.

\section{Preliminary materials}\label{sec:pre}
Let us begin with some notation and definitions used throughout this paper. Firstly, notice that we will follow the usual convention of denoting by $C$ a general positive constant which will not necessarily be the same at different occurrences and which can also change from line to line. Peculiar dependencies on parameters will be emphasized in parentheses when needed, i.e. $C=C(n,p)$ means that $C$ depends on $n,p$. In the sequel, let $\Omega$ be an open bounded domain in $\mathbb{R}^n$ with smooth boundary $\partial\Omega$, $n \ge 2$. We denote by $|E|$ the finite $n$-dimensional Lebesgue measure of a certain measurable set $E \subset \mathbb{R}^n$ and with $h: E \to \mathbb{R}$ being a measurable map, we shall denote by
\begin{align*}
\fint_{E} h(x) dx = \frac{1}{|E|} \int_{E} h(x) dx
\end{align*}
its mean integral on $E$. For any $p>1$, the notation $p'$ stands for H\"older conjugate exponent of $p$, i.e. $p'=\frac{p}{p-1}$. Further, we shall make use of the star notation, $p^*$ for the Sobolev critical exponent to $p$
\begin{align*}
p^* = \begin{cases}
\frac{np}{n-p}, \quad \text{if} \ \ p<n\\
+\infty, \quad \text{if} \ \ p \ge n.
\end{cases}
\end{align*}

For the convenience of the reader, we recall some basic definitions and properties related to the Young function and generalized Orlicz-Sobolev spaces as follows. For a quite comprehensive account, the interested reader might start by  referring to~\cite{HH19} and the bibliography therein.

\begin{definition}[Young function]
We say that $G: [0,\infty) \to [0,\infty)$ to be a Young function if $G$ is non-decreasing convex and satisfies the following conditions
\begin{align*}
G(0) = 0, \ \lim_{t \to \infty} G(t) = \infty, \ \lim_{t \to 0^+} \frac{G(t)}{t} = 0, \ \lim_{t \to \infty} \frac{G(t)}{t} = \infty.
\end{align*}
Then, the complementary Young function $G^*$ to $G$ is defined by
\begin{align*}
G^*(t) = \sup \{ts - G(s): \ s \ge 0\}, \quad t \ge 0.
\end{align*}
\end{definition}

\begin{definition}[$\Delta_2$ and $\nabla_2$-conditions]
We say that the Young function $G$ satisfies $\Delta_2$ condition if there exists $\tau_1>1$ such that
\begin{align*}
G(2t) \le \tau_1 G(t), \ \mbox{ for all } t \ge 0.
\end{align*}
In this case, we write $G \in \Delta_2$. We say that the Young function $G$ satisfies $\nabla_2$ condition, denoted by $G \in \nabla_2$, if there exists $\tau_2>1$ such that
\begin{align*}
G(t) \le \frac{1}{2\tau_2} G(\tau_2 t), \ \mbox{ for all } t \ge 0.
\end{align*}
It is worth noting that if $G$ satisfies both $\Delta_2$ and $\nabla_2$-conditions, we shall write $G \in \Delta_2 \cap \nabla_2$.
\end{definition}

In what follows, we state a useful lemma for later use in the proof, interested readers might consult~\cite{HH19} for further details. 

\begin{lemma} \label{lem-Young} Let $G$ be a Young function.
\begin{itemize}
\item[i)] $G^*$ is also a Young function and $(G^*)^* = G$.
\item[ii)] $G \in \nabla_2 \Leftrightarrow G^* \in \Delta_2$.
\item[iii)] If $G \in \Delta_2$ then there exist some constants $1<\nu_1 \le \nu_2$ and $C>1$ independent of $\sigma$ and $t$ such that
\begin{align*}
C^{-1} \min\{\sigma^{\nu_1},\sigma^{\nu_2}\} G(t) \le G(\sigma t) \le C \max\{\sigma^{\nu_1},\sigma^{\nu_2}\} G(t), \quad \forall t, \sigma \ge 0.
\end{align*}
\item[iv)] If $G \in \Delta_2 \cap \nabla_2$ then 
\begin{align}\label{G-star}
G^*\left(\frac{G(t)}{t}\right) \le G(t) \le G^*\left(\frac{2G(t)}{t}\right) , \quad \mbox{ for all } t >0.
\end{align}
Moreover, for every $\varepsilon \in (0,1)$, there exists $C>0$ such that
\begin{align*}
st \le \varepsilon G(s) + C G^*(t), \quad \mbox{ for all } s, t \ge 0.
\end{align*}
\item[v)] If $G \in \Delta_2 \cap \nabla_2$ then for every $\varepsilon \in (0,1)$, there exists $C>0$ such that
\begin{align}\label{Young-ineq}
s \frac{G(t)}{t} \le \varepsilon G(s) + C G(t), \quad \mbox{ for all } s \ge 0, t > 0.
\end{align}
\end{itemize}
\end{lemma}

Regarding this study, we also collect definitions of Orlicz and Orlicz-Sobolev spaces more or less known (cf.~\cite{LM2019}). 

\begin{definition}[Orlicz spaces]
Let $\Omega$ be a bounded open subset in $\mathbb{R}^n$ and $G: [0,\infty) \to [0,\infty)$ be a Young function. The Orlicz class $O^{G}(\Omega)$ is defined to be the set all of measurable functions $f:\Omega \to \mathbb{R}$ satisfying 
$$\int_{\Omega} G(|f(x)|)dx < \infty.$$ 
The Orlicz space $L^G(\Omega)$ is the smallest linear space containing $O^{G}(\Omega)$, endowed with the following Luxemburg norm
\begin{align*}
\|f\|_{L^G(\Omega)} = \inf\left\{\sigma >0 : \ \fint_{\Omega} G\left(\frac{|f(x)|}{\sigma}\right)dx \le 1 \right\}.
\end{align*}
\end{definition}


\begin{definition}[Orlicz-Sobolev spaces]
Let $G: [0,\infty) \to [0,\infty)$ be a Young function. 
We denote by $W^{1,G}(\Omega)$ the Orlicz-Sobolev spaces which is the set of all measurable functions $f \in L^{G}(\Omega)$ such that $|\nabla f| \in L^{G}(\Omega)$. In this space, we consider the norm
\begin{align*}
\|f\|_{W^{1,G}(\Omega)} = \|f\|_{L^{G}(\Omega)} + \|\nabla f\|_{L^{G}(\Omega)}, 
\end{align*}
where we denote $\|\nabla f\|_{L^{G}(\Omega)} = \||\nabla f|\|_{L^{G}(\Omega)}$. Moreover, we denote by $W_0^{1,G}(\Omega)$ the closure of $C_0^{\infty}(\Omega)$ in $W^{1,G}(\Omega)$.
\end{definition}

\begin{lemma} \label{lem-Hold} Let $G \in \Delta_2 \cap \nabla_2$ be a Young function.
\begin{itemize}
\item[i)] Then $L^G(\Omega) = O^{G}(\Omega)$ and $(L^G(\Omega), \|\cdot\|_{L^G(\Omega)})$ is a reflexive Banach space.
\item[ii)] If $f \in L^G(\Omega)$ and $g \in L^{G^*}(\Omega)$ then $fg \in L^1(\Omega)$ and there holds
\begin{align}\label{Hold-ineq}
\fint_{\Omega} |fg|dx \le 2 \|f\|_{L^G(\Omega)} \|g\|_{L^{G^*}(\Omega)}.
\end{align}
\end{itemize}
\end{lemma}

The next part provides us with definitions and some necessary auxiliary results for the Orlicz-Zygmund and Orlicz-Zygmund-Sobolev spaces which will be employed later on. We refer here to~\cite{BS1988, Stein, IV99} for more details concerning these spaces.

\begin{definition}[Orlicz-Zygmund spaces $L^p\log^{\alpha} L(\Omega)$]
\label{def:O-Z-sp}
Let either $\alpha>0$ and $1 \le p <\infty$, the Orlicz-Zygmund space $L^p\log^{\alpha} L(\Omega)$ is defined by
\begin{align}\label{def:Llog}
L^p\log^{\alpha} L(\Omega) = \left\{f \in L^1(\Omega): \ \fint_{\Omega} |f(x)|^p \log^{\alpha} (e+|f(x)|)dx < \infty \right\},
\end{align}
with the following finite Luxemburg norm
\begin{align}\label{def:normLlog}
\|f\|_{L^p\log^{\alpha} L(\Omega)} = \inf\left\{\sigma>0: \ \fint_{\Omega} \frac{|f(x)|^p}{\sigma^p} \log^{\alpha} \left(e+\frac{|f(x)|}{\sigma}\right)dx \le 1\right\}.
\end{align}
\end{definition}
\begin{remark} $ $
\begin{itemize}
\item If $\alpha=0$, then the space $L^p\log^{\alpha} L(\Omega)$ reduces to the classical Lebesgue space $L^p(\Omega)$ whose norm is simply denoted by $\|\cdot\|_p$.
\item If $p=1, \alpha>0$, we concern about the  Zygmund class. In what follows, when $p=1$ or $\alpha=1$, we shall write $L\log^{\alpha} L(\Omega)$ or $L^p\log L(\Omega)$, respectively for simplicity.
\item  Let $f \in  L^p\log^{\alpha} L(\Omega)$ such that 
\begin{align*}
\|f\|_{p} := \left(\fint_{\Omega} |f(x)|^p dx \right)^{\frac{1}{p}} >0,
\end{align*}
we shall denote by
\begin{align}\label{modula}
[f]_{L^p\log^{\alpha} L(\Omega)}:= \left(\fint_{\Omega} |f(x)|^p  \log^{\alpha} \left(e+\frac{|f(x)|}{\|f\|_{p}}\right)dx\right)^{1/p}
\end{align}
the modular function of $f$.
\end{itemize}
\end{remark}

\begin{remark}
Within the setting of Orlicz-Zygmund spaces, we observe that it is not easy to work with the Luxemburg norm of the Orlicz-Zygmund space $L^p\log^{\alpha} L(\Omega)$ in~\eqref{def:Llog}. However, the interesting point here is that one can obtain the equivalence between the Luxemburg norm $\|\cdot\|_{L^p\log^{\alpha} L(\Omega)}$ in~\eqref{def:normLlog} and the modular function $[\, \cdot\,]_{L^p\log^{\alpha} L(\Omega)}$~\eqref{modula}. The following lemma reveals the relationship between the norm and modular function of $f$.
\end{remark}

\begin{lemma}\label{lem:Llog}
For every $f \in  L^p\log^{\alpha} L(\Omega)$ with $p \ge 1$ and $\alpha>0$, one has
\begin{align}\label{est:Llog}
 \|f\|_{L^p\log^{\alpha} L(\Omega)} \le [f]_{L^p\log^{\alpha} L(\Omega)} \le  \left[2^{\alpha}+\left(\frac{2\alpha}{ep}\right)^{\alpha}\right] \|f\|_{L^p\log^{\alpha} L(\Omega)}.
\end{align}
For $q>p$, then there exists a constant $C=C(p,q,\alpha)>0$ such that
\begin{align}\label{est-lem:log}
[f]_{L^p\log^{\alpha} L(\Omega)} = \left(\fint_{\Omega} |f(x)|^p  \log^{\alpha} \left(e+\frac{|f(x)|}{\|f\|_{p}}\right)dx\right)^{1/p} \le C \left(\fint_{\Omega} |f(x)|^{q}dx\right)^{\frac{1}{q}},
\end{align}
for every $f \in L^q(\Omega)$.  In particular, there holds
\begin{align}\label{re-lem-log}
L^q(\Omega) \subset L^p \log^{\alpha} L(\Omega) \subset L^p(\Omega).
\end{align}
\end{lemma}
\begin{proof}
Let $f \in  L^p\log^{\alpha} L(\Omega)$ such that $\|f\|_{p} > 0$, we set $\sigma = \|f\|_{L^p\log^{\alpha} L(\Omega)}$. It is clear to see that 
\begin{align}\label{est-LL-0}
\sigma = \left[\fint_{\Omega} \left|f(x)\right|^p \log^{\alpha} \left(e+\frac{|f(x)|}{\sigma}\right)dx\right]^{\frac{1}{p}} \ge \|f\|_{p},
\end{align}
which implies
\begin{align*}
\sigma \le \left[\fint_{\Omega} \left|f(x)\right|^p \log^{\alpha} \left(e+\frac{|f(x)|}{\|f\|_{p}}\right)dx\right]^{\frac{1}{p}} = [f]_{L^p\log^{\alpha} L(\Omega)}.
\end{align*}
Applying the following fundamental inequality
\begin{align}\label{tge1}
\log^{\alpha} (e+st) \le 2^{\alpha-1} \big[ \log^{\alpha} (e+s) + \log^{\alpha} t\big], \quad s \ge 0, \ t \ge 1,
\end{align}
one obtains that
\begin{align}\label{est-LL-1}
[f]_{L^p\log^{\alpha} L(\Omega)}^p & =  \fint_{\Omega} |f(x)|^p \log^{\alpha} \left(e+\frac{|f(x)|}{\|f\|_{p}}\right)dx \notag \\
& \le 2^{\alpha-1} \left[ \fint_{\Omega} |f(x)|^p \log^{\alpha} \left(e+\frac{|f(x)|}{\sigma}\right)dx +\fint_{\Omega} \left|f(x)\right|^p \log^{\alpha} \left(\frac{\sigma}{\|f\|_{p}}\right)dx\right].
\end{align}
At this stage, we are allowed to make use of the following inequality
\begin{align}\label{tge1-2}
\log^{\alpha} t \le \left(\frac{\alpha}{ep}\right)^{\alpha} t^p, \quad \mbox{for all } t \ge 1
\end{align}
and from~\eqref{est-LL-1}, it yields that
\begin{align}\notag
[f]_{p\log}^p & \le 2^{\alpha-1}\left[ \sigma^p + \left(\frac{\alpha}{ep}\right)^{\alpha} \fint_{\Omega} \left|f(x)\right|^p \frac{\sigma^p}{\|f\|_{p}^p} dx \right] \le  2^{\alpha p} \left[1+\left(\frac{\alpha}{ep}\right)^{\alpha}\right]^p \sigma^p,
\end{align}
and therefore, we evoke this estimate to conclude~\eqref{est:Llog}. We also send the reader to~\cite{IV99} for considerations and references concerning the case when $\alpha=1$.  

In a next step, we prove~\eqref{est-lem:log}. Note that in~\cite{AM05}, the authors dealt with the case when $p=1$. So, let us now prove it for the general case when $p \ge 1$. Using the preceding inequality~\eqref{tge1-2}, we arrive at
\begin{align*}
\log^{\alpha} (e+t) \le \left(\frac{\alpha}{ep}\right)^{\alpha} (e+t)^p \le \left(\frac{\alpha}{ep}\right)^{\alpha}  2^{p-1} e^p \left(1 + t^p\right), 
\end{align*}
for all $t \ge 0$ and $\alpha>0$. Combining H\"older's inequality with it, there holds
\begin{align}
 \fint_{\Omega} |f(x)|^p & \log^{\alpha} \left(e+\frac{|f(x)|}{\|f\|_{p}}\right)dx \notag \\
& \le  \left(\fint_{\Omega} |f(x)|^q dx \right)^{\frac{p}{q}} \left(\fint_{\Omega} \log^{\frac{q\alpha}{q-p}} \left(e+\frac{|f(x)|}{\|f\|_{p}}\right)dx \right)^{1-\frac{p}{q}} \notag
\\ & \le C(p,q,\alpha) \left(\fint_{\Omega} |f(x)|^q dx \right)^{\frac{p}{q}} \left(1+\fint_{\Omega} \frac{|f(x)|^p}{\|f\|_{p}^p}dx \right)^{1-\frac{p}{q}} \notag \\
& \le C(p,q,\alpha) \left(\fint_{\Omega} |f(x)|^q dx \right)^{\frac{p}{q}}, \notag
\end{align}
which deduces to~\eqref{est-lem:log}. Finally, the last relation~\eqref{re-lem-log} can be directly obtained from~\eqref{est-LL-0} and~\eqref{est-lem:log}. The proof is complete.
\end{proof} 

\begin{remark}\label{rmk-1}
Note at this point, that since there exist two positive constants $C_1, C_2$ such that the estimate
 \begin{align*}
C_1(\gamma) \log(e+t) \le \log(e+t^{\gamma}) \le C_2(\gamma) \log(e+t),
\end{align*}
holds for every $t, \gamma \ge 0$, and hence 
\begin{align*}
f \in L^q\log L(\Omega) \Leftrightarrow |f|^{\frac{q}{p}} \in L^p\log L(\Omega).
\end{align*}
for all $p,q \ge 1$. 
\end{remark}

The kind of results we are interested in here can be considered an analog argumentation to the above. Given $p > 1$, let us consider the Orlicz-Zygmund space $\mathbb{L}^{p\log}(\Omega)$ defined by
\begin{align}\label{def:Lbb}
\mathbb{L}^{p\log}(\Omega) = \left\{f \in L^1(\Omega): \ \fint_{\Omega} |f(x)|^p + |f(x)|^p \log (e+|f(x)|)dx < \infty \right\},
\end{align}
with the following Luxemburg norm
\begin{align}\label{def:normLbb}
\|f\|_{\mathbb{L}^{p\log}(\Omega)} = \inf\left\{\sigma>0: \ \fint_{\Omega} \frac{|f(x)|^p}{\sigma^p} + \frac{|f(x)|^p}{\sigma^p} \log \left(e+\frac{|f(x)|}{\sigma}\right)dx \le 1\right\}.
\end{align}
Assume that $f \in  \mathbb{L}^{p\log}(\Omega)$ satisfying  $\|f\|_{p} > 0$, we shall denote by
\begin{align}\label{modula_Lbb}
[f]_{\mathbb{L}^{p\log}(\Omega)} := \left[\fint_{\Omega} |f(x)|^p + |f(x)|^p  \log \left(e+\frac{|f(x)|}{\|f\|_{p}}\right)dx\right]^{1/p}
\end{align}
as the modular function of $f$.

In a similar fashion as in Lemma~\ref{lem:Llog}, it makes sense to prove that the following relation between Luxemburg norm $\|\cdot\|_{\mathbb{L}^{p\log}(\Omega)}$ in~\eqref{def:normLbb} and the modular function $[\, \cdot \,]_{\mathbb{L}^{p\log}(\Omega)}$ in~\eqref{modula_Lbb} holds (in Lemma~\ref{lem:Lbb} below).

\begin{lemma}\label{lem:Lbb}
For every $f \in  \mathbb{L}^{p\log}(\Omega)$ with $p > 1$, one has
\begin{align}\label{est:Lbb}
 \|f\|_{\mathbb{L}^{p\log}(\Omega)} \le [f]_{\mathbb{L}^{p\log}(\Omega)} \le 4\|f\|_{\mathbb{L}^{p\log}(\Omega)}.
\end{align}
\end{lemma}

\begin{definition}[Orlicz-Zygmund-Sobolev spaces] \label{def-OZS-space}
The Orlicz-Zygmund-Sobolev spaces, denoted by $W^{1,p\log}(\Omega)$, is the set f all measurable functions $f \in \mathbb{L}^{p\log}(\Omega)$ such that $|\nabla f| \in \mathbb{L}^{p\log}(\Omega)$. This function space is equipped with the norm
\begin{align*}
\|f\|_{W^{1,p\log}(\Omega)} = \|f\|_{\mathbb{L}^{p\log}(\Omega)} + \|\nabla f\|_{\mathbb{L}^{p\log}(\Omega)}.
\end{align*}
\end{definition}
When not important, or clear from the context, and for the sake of brevity, we shall write $\|\nabla f\|_{\mathbb{L}^{p\log}(\Omega)}$ in place of $\||\nabla f|\|_{\mathbb{L}^{p\log}(\Omega)}$. Moreover, one has the fact that $W_0^{1,p\log}(\Omega)$ coincides with the closure of  $C_0^{\infty}(\Omega)$ in $W^{1,p\log}(\Omega)$. For ease of notation, we also write
\begin{align*}
\mathbb{W}:= W^{1,p\log}(\Omega)\quad  \text{ and } \quad \mathbb{W}_0 = W_0^{1,p\log}(\Omega).
\end{align*}
Here, in $\mathbb{W}_0$, let us consider the following norm
\begin{align}\label{norm-W0}
\|f\|_{\mathbb{W}_0} = \|\nabla f\|_{\mathbb{L}^{p\log}(\Omega)}, \quad f \in \mathbb{W}_0.
\end{align}

\begin{remark}
We emphasize here that the Orlicz-Zygmund space $L^p\log L(\Omega)$ in~\eqref{def:Lbb} can be defined as the Orlicz space $L^{\varphi_p}(\Omega)$ associated to the Young function $\varphi_p: \mathbb{R}^+ \to \mathbb{R}^+$ given by
\begin{align}\label{def:varphi-p}
\varphi_p(t) = t^p \log(e+t), \quad t \ge 0.
\end{align}
Similarly, $\mathbb{W}$ and $\mathbb{W}_0$ can be defined as the Orlicz-Sobolev space $W^{1,H_p}(\Omega)$ and $W_0^{1,H_p}(\Omega)$ respectively, where the Young function $H_p: \mathbb{R}^+ \to \mathbb{R}^+$ is given by
\begin{align}\label{def:H}
H_p(t) = t^p + t^p \log(e+t), \quad t \ge 0.
\end{align}
By this way, since $\varphi_p, H_p \in \Delta_2 \cap \nabla_2$ with $p>1$, then all spaces $L^p\log L(\Omega)$, $\mathbb{L}^{p\log}(\Omega)$, $\mathbb{W}$ and $\mathbb{W}_0$ are reflexive Banach spaces. 
\end{remark}

Regarding Orlicz-Zygmund-Sobolev spaces more or less known, at this stage, we will apply \cite[Theorem 1, Theorem 3]{C96} to obtain the following interesting compact embedding result.
\begin{lemma}\label{lem:embed}
For every $1<q \le p^* := \frac{np}{n-p}$, the embedding 
\begin{align}\label{embed-comp}
\mathbb{W}_0 \hookrightarrow L^{q}\log L(\Omega)
\end{align}
is compact. Moreover, there exists a constant $C=C(p,q)>0$ such that
\begin{align}\label{embed-ineq}
\|f\|_{L^q\log L(\Omega)} & \le C \|\nabla f\|_{\mathbb{L}^{p\log}(\Omega)}.
\end{align}
As a consequence result, we obtain 
\begin{align}\label{equiv-norm}
\|f\|_{\mathbb{W}_0} \le \|f\|_{\mathbb{W}} \le C \|f\|_{\mathbb{W}_0}.
\end{align}
\end{lemma}
\begin{proof}
For every Young function $H$ belonging to $\Delta_2 \cap \nabla_2$, the complementary function $H^*$ to $H$ is defined by
\begin{align*}
H^*(t) = \sup \{ts - H(s): \ s \ge 0\}, \quad t \ge 0.
\end{align*}
Let us now introduce the following Young function
\begin{align*}
A_{n,H}(t) = \int_0^t s^{\frac{n}{n-1}-1} \left[\Phi_{n,H}^{-1}(s^{\frac{n}{n-1}})\right]^{\frac{n}{n-1}}ds,
\end{align*}
where $\Phi_{n,H}^{-1}$ is the right-continuous generalized inverse function of $\Phi_{n,H}$ defined by
\begin{align*}
\Phi_{n,H}(t) = \int_0^t \frac{H^*(s)}{s^{1+\frac{n}{n-1}}} ds, \ \mbox{ and } \ \Phi_{n,H}^{-1}(t) = \inf\{s \ge 0: \ \Phi_{n,H}(s)>t\}, 
\end{align*}
for $t \ge 0$. Thanks to \cite[Theorem 3]{C96}, we infer that
\begin{align*}
W_0^{1,H}(\Omega) \hookrightarrow L^{A_{n,H}}(\Omega)
\end{align*}
is a continuous embedding. Moreover, if $B \in \Delta_2 \cap \nabla_2$ is a Young function that more slowly than $A_{n,H}$ then the following embedding
\begin{align*}
W_0^{1,H}(\Omega) \hookrightarrow L^{B}(\Omega)
\end{align*}
is compact. The proof of the embedding~\eqref{embed-comp} is complete if we may apply the result above with function $H = H_p$ and $B = \varphi_q$ given as in~\eqref{def:H} and~\eqref{def:varphi-p} respectively. In fact, the most important point is to verify that $B$ is more slowly than $A_{n,H}$ in the following sense 
\begin{align}\label{A-B}
\lim_{t\to \infty} \frac{A_{n,H}(t)}{B(t)} = \infty.
\end{align}
Indeed, let us consider $G = \varphi_p$. Since $H(t) \ge  G(t)$ for every $t \ge 0$, one has $H^*(t) \le G^*(t)$. It leads to $\Phi_{n,H}(t) \le \Phi_{n,G}(t)$ and then $\Phi_{n,H}^{-1}(t) \ge \Phi_{n,G}^{-1}(t)$ for each $t \ge 0$. It follows that $A_{n,H}(t) \ge A_{n,G}(t)$ for all $t \ge 0$. On the other hand, see \cite[Example 1]{C96}, it is well known that
$$A_{n,G}(t) = Ct^{p^*} \log^{\frac{p^*}{p}}(e+t), \quad t \ge 0.$$ 
Therefore, we have
\begin{align}\notag
\lim_{t\to \infty} \frac{A_{n,G}(t)}{B(t)} =  C \lim_{t\to \infty} t^{p^*-q} \log^{\frac{p^*}{p}-1}(e+t) = \infty, \quad \mbox{ for all } q \in (1,p^*],
\end{align}
which guarantees~\eqref{A-B}. Hence, it allows us to conclude that $B$ is more slowly than $A_{n,H}$. We remark here that we may only infer that the embedding 
$$W_0^{1,H}(\Omega) \hookrightarrow L^{p^*}\log^{\frac{p^*}{p}}L(\Omega)$$ 
is continuous. Moreover, applying \cite[Theorem 1]{C96} again, we obtain inequality~\eqref{embed-ineq}. In particular, applying this embedding result with $q=p$, there exists a constant $C>0$ such that
\begin{align}\label{embed-ineq-2}
\|f\|_{L^p\log L(\Omega)} & \le C \|\nabla f\|_{\mathbb{L}^{p\log}(\Omega)}.
\end{align}
On the other hand, thanks to inequalities~\eqref{est:Lbb} in Lemma~\ref{lem:Lbb} and~\eqref{est:Llog} in Lemma~\ref{lem:Llog}, one gets that
\begin{align}\label{ssLlog}
\|f\|_{\mathbb{L}^{p\log}(\Omega)} & \le [f]_{\mathbb{L}^{p\log} (\Omega)}  = \left[\fint_{\Omega} |f(x)|^p + |f(x)|^p  \log \left(e+\frac{|f(x)|}{\|f\|_{p}}\right)dx\right]^{1/p} \notag \\
& \le 2^{\frac{1}{p}} \left[\fint_{\Omega} |f(x)|^p  \log \left(e+\frac{|f(x)|}{\|f\|_{p}}\right)dx\right]^{1/p} = 2^{\frac{1}{p}} [f]_{L^p\log L(\Omega)} \notag\\
& \le 8 \|f\|_{L^p\log L(\Omega)}.
\end{align}
Combining two estimates in~\eqref{embed-ineq-2} and~\eqref{ssLlog}, we have
\begin{align*}
\|f\|_{\mathbb{W}_0} = \|\nabla f\|_{\mathbb{L}^{p\log}(\Omega)} \le \|f\|_{\mathbb{W}} = \|f\|_{\mathbb{L}^{p\log}(\Omega)} + \|\nabla f\|_{\mathbb{L}^{p\log}(\Omega)} \le C \|f\|_{\mathbb{W}_0}.
\end{align*}
The proof is complete.
\end{proof}

This will not affect the rest, we conclude this section by briefly recalling some definitions that will be required later for the proof of our main results. 

\begin{definition}
Let $\mathbb{W}$ be a reflexive Banach space endowed with the norm $\|\cdot\|_\mathbb{W}$. We denote by $\mathbb{W}^*$ its topological dual space of $\mathbb{W}$ and by $\langle \cdot,\cdot \rangle$ the duality pairing between $\mathbb{W}$ and $\mathbb{W}^*$. The norm convergence is denoted by $\to$ and the weak convergence by $\rightharpoonup$. Consider a nonlinear map $A: \mathbb{W} \to \mathbb{W}^*$, then
\begin{enumerate}
\item[(i)] $A$ is called bounded if it maps bounded sets to bounded sets;
\item[(ii)] $A$ is said to be coercive if there holds
\begin{align*}
\lim_{\|u\|_\mathbb{W}}{\frac{\langle Au,u \rangle}{\|u\|_\mathbb{W}}} = +\infty;
\end{align*}
\item[(iii)] $A$ is called pseudo-monotone if $u_n\rightharpoonup u$ in $\mathbb{W}$ and $\limsup_{n\to +\infty}{\langle Au_n,u_n-u\rangle} \le 0$, it follows that
\begin{align*}
\langle Au,u-w \rangle \le \liminf_{n\to \infty}{\langle Au_n,u_n-w \rangle}, \quad \forall w \in \mathbb{W};
\end{align*}
\item[(iv)] $A$ is said to satisfy the $(S_+)$-property if $u_n \rightharpoonup u$ in $\mathbb{W}$ and $\limsup_{n\to +\infty}{\langle Au_n,u_n-u\rangle} \le 0$ imply $u_n \to u$ in $\mathbb{W}$.
\end{enumerate}
\end{definition}

\section{Statement of existence theorem}
\label{sec:main}

In this section, we shall present the main result, whose proof will be made up in the last section. As already alluded to in the introduction, we will focus our attention in this paper on the existence of at least a weak solution to~\eqref{main-eq}. 

The solution of problem~\eqref{main-eq} is understood in the weak sense. We say that a weak solution to~\eqref{main-eq} is a map $u \in \mathbb{W}_0$ that satisfying the weak formulation
\begin{align}\label{var-form}
\int_{\Omega} \mathcal{A}(\nabla u) \cdot \nabla v dx = \int_{\Omega} F(x,u,\nabla u) v dx,
\end{align} 
for every test function $v \in \mathbb{W}_0$, where $\mathbb{W}_0$ is the Orlicz-Zygmund-Sobolev space shown in Definition~\ref{def-OZS-space}.

Before going into details, let us briefly discuss the eigenvalue problem of the $p$-Laplace operator for $1<p<\infty$ on a bounded domain $\Omega$ with homogeneous Dirichlet boundary condition:
\begin{align}\label{p-Lap}
\begin{cases} -\mathrm{div}(|\nabla u|^{p-2}\nabla u) & = \ \lambda |u|^{p-2}u, \quad \mbox{ in } \Omega, \\ \hspace{1.5cm} u & = \ 0, \hspace{1.6cm} \mbox{ on } \partial \Omega. \end{cases}
\end{align}

There are several situations that we are interested in knowing about the first eigenvalue, say $\lambda_{1,p}$ of problem~\eqref{p-Lap}. The first eigenvalue has some well-known properties: positive, simple, and isolated, (see for e.g.~\cite{L06}). In addition, the following inequality
\begin{align}\label{eig-value}
\int_{\Omega}|u(x)|^p dx \le \frac{1}{\lambda_{1,p}} \int_{\Omega}|\nabla u(x)|^p dx , 
\end{align}
holds for every $u \in W_0^{1,p}(\Omega)$. We refer the reader to~\cite{Lin90, GP1987} for further references.

Once having all the preparations at hand, it allows us to state the main existence result of this paper.  

\begin{theorem}\label{theo-main}
Let  $n \ge 2$, $1<p<n$ and $\lambda_{1,p}$ be the first eigenvalue of problem~\eqref{p-Lap}. Provided the following parameters that satisfying 
\begin{align}\label{cond-mu}
\mu_1, \mu_2>0, \quad 0< \mu_3 < p \lambda_{1,p}, \quad \ 0 < \mu_4 < p, \quad \mbox{and } 1< q \le \frac{np}{n-p}.
\end{align}
Assume further that the nonlinearity on the right-hand side of~\eqref{main-eq}, $F: \Omega \times \mathbb{R} \times \mathbb{R}^n \to \mathbb{R}$, is a Carath\'eodory function and there exist two functions $g \in (L^{q}\log L(\Omega))^*$ and $h \in L^1(\Omega)$ satisfying the following conditions
\begin{align}\label{Asump-f-1}
\begin{cases}
|F(x,t,\xi)| \le g(x) + \mu_1 |t|^{q-1} \log (e+|t|) + \mu_2 |\xi|^{\frac{p(q-1)}{q}} \log \left(e+|\xi|^{\frac{q}{p}}\right),  \\ 
F(x,t,\xi) t \le h(x) + \mu_3 |t|^p  + \mu_4 |\xi|^p \log(e + |\xi|),  \end{cases}
\end{align}
for a.e. $x \in \Omega$, for all $t \in \mathbb{R}$ and $\xi \in \mathbb{R}^n$. Then, the equation~\eqref{main-eq} admits at least one weak solution $u \in \mathbb{W}_0$.
\end{theorem}

\begin{remark}
Theorem~\ref{theo-main} can be extended in a natural way in the study of a class of unbalanced double-phase problems with variable exponent, i.e. when $p \equiv p(\cdot)$, especially when we deal with obstacle models (models that combine a double-phase operator along with an obstacle constraint). As far as we know, models with variable exponents have been the object of intensive interest in the applied sciences. For instance, non-Newtonian fluids that change their viscosity when the electromagnetic field is present, image segmentation, etc. Generalizing the existence result to this case therefore is significantly more involved, and will be addressed in the future.

Also, similar to what has been done with quasilinear elliptic equations/systems driven by double-phase operators involving two distinct power hardening exponents $p$ and $q$, is of the type
\begin{align*}
-\mathrm{div}(|\nabla u|^{p-2}\nabla u +a(x)|\nabla u|^{q-2}\nabla u) = F(x,u,\nabla u), \quad \text{in} \ \ \Omega,
\end{align*}
that treated recently by many authors~\cite{EMW2022, GW19, LD18, MW2019, VW2022}, etc, it would be very interesting to deal with the existence of solutions for problem~\eqref{main-eq} or the Euler-Lagrange equation of the functional in~\eqref{eq:G} under different types of boundary conditions (Dirichlet, Neumann, and mixed).
\end{remark}

\begin{remark}
From a mathematical point of view, the authors not only devote special attention to the question of clarifying the existence of the weak solutions/minimizers to certain functionals but also the regularity theory has been paid a lot of attention recently. The regularity study of non-uniformly elliptic equations with $(p,q)$-growth has been widely investigated. For example, we mention some famous results in regularity theory as in~\cite{MS1999, AM01, AM05, BCM15, BCM16, BCM18, BM2020, CM2015, CM2015_2, FM2019, FM2020, FM2020_2, TN21, TNPD22,  BO2017, Marcellini1989, Marcellini1991} and references therein. Since the regularity for borderline double-phase problems is not yet completely settled, it is natural to establish regularity properties for minimizers/solutions to variational integrals and related equations. 
\end{remark}

\begin{remark}
It is to be noticed that in the context of general functionals with double-phase and related equations whose integrand grows almost linearly with respect to the gradient, the existence results are a much more significant challenge. In a notable work by Fillipis and Mingione~\cite{FM2023}, authors dealt with an initial regularity result for the minimizers of functionals of ``nearly-linear'' type
\begin{align*}
\omega \mapsto \mathcal{I}(\omega,\Omega):=\int_\Omega{\left(a(x)|\nabla u|\mathrm{log}(1+|\nabla u|)+b(x)|\nabla u|^q\right)dx}.
\end{align*}
This model is the combination between the classical nearly linear growth and $q$-power growth. Many open questions on both the existence and regularity of such problems remain ahead of us. We expect that the same result would be extended to solutions to nonlinear problems in such a borderline case.
\end{remark}

\section{Proof of main result}
\label{sec:proof}

This section sets forth the proof of Theorem~\ref{theo-main} with all the preceding results at hand. Further, the proof of our theorem is based on the subjectivity theorem for pseudo-monotone operators as follows. To explore more, we refer to~\cite[Theorem 2.99]{CLM07}.

\begin{theorem}\label{theo-CLM}
Let $\mathbb{W}$ be a real reflexive Banach space. Assume that $\mathcal{G}: \mathbb{W} \to \mathbb{W}^*$ is a bounded, pseudo-monotone and coercive operator. Then the equation $\mathcal{G}(u) = 0$ admits at least one solution in $\mathbb{W}$.
\end{theorem}

Although the key idea of the proof follows a similar path to the proof for double-phase problems with $(p,q)$-growth in~\cite{GW19}, our problem~\eqref{main-eq} (known as the borderline case of a double-phase problem) is difficult to handle by itself nature. Due to the logarithmic growth, the technique is not the same as in~\cite{GW19}.  As a natural way, from~\eqref{var-form}, the operator $\mathcal{G}$ mentioned in Theorem~\ref{theo-CLM} will be defined by
\begin{align}\label{def-G}
\langle \mathcal{G}(u),v \rangle = \int_{\Omega} \left[ \mathcal{A}(\nabla u) \cdot \nabla v - F(x,u,\nabla u) v \right] dx, \quad v \in \mathbb{W}_0.
\end{align}

The first difficulty comes from the well-definedness of the operator $\mathcal{G}$, which is obtained by the construction of compact embedding result from $\mathbb{W}_0$ into another reflexive Banach space. Of course, if we embed $\mathbb{W}_0$ into Lebesgue spaces, then assumptions on the convection term $F$ may be not optimal (see assumption~\eqref{Asump-f-1}). To obtain a better result corresponding to the condition of $F$, we connect to a compact embedding in Orlicz-Zygmund spaces $L^q\log L(\Omega)$, which is an optimal embedding result as in our knowledge until now. This embedding result can be obtained from a sharp theorem proved by Cianchi in~\cite{C96}.

Furthermore, when we work with the Luxemburg norm of Orlicz-Zygmund space $L^q\log L(\Omega)$, we often consider another integral term that is related to the modular function. Therefore, the proof of coercive property of $\mathcal{G}$ defined in~\eqref{def-G} becomes more difficult with this type of modular function.

In order to obtain the pseudo-monotone property of $\mathcal{G}$, one needs to show the $(S_+)$ condition of the operator related to the term on the left-hand side of~\eqref{var-form} that involves the logarithmic perturbation. In this paper, we specifically use a very different method to the one in~\cite{LD18} and~\cite{GW19}.

At this stage, let us consider the following double-phase operator $\mathcal{L}: \mathbb{W}_0 \to \mathbb{W}_0^*$ defined by
\begin{align}\label{def:LW}
\langle \mathcal{L}(u),v\rangle &= \int_{\Omega} \mathcal{A}(\nabla u) \cdot \nabla v dx, 
\end{align}
for every $u,v \in \mathbb{W}_0$. \\

The next lemma ensures some important properties of the operator $\mathcal{L}$ that play a crucial role in the main proof of existence. 

\begin{lemma}\label{lem:properties-L}
Operator $\mathcal{L}$ defined in~\eqref{def:LW} is continuous, bounded, monotone and it satisfies $(S)_+$ condition.
\end{lemma}
\begin{proof}
It is easy to see that $\mathcal{L}$ is continuous. Let us now show that $\mathcal{L}$ is bounded. Indeed, let us set $\tilde{u} = u/\|u\|_{\mathbb{W}_0}$ and $\tilde{v} = v/\|v\|_{\mathbb{W}_0}$. One has
\begin{align}\label{est-11}
\left|\left\langle \mathcal{L}\left(\tilde{u}\right), \tilde{v}\right\rangle\right| & = \left|p\int_{\Omega} \left[ \left|\nabla \tilde{u}\right|^{p-2} \nabla \tilde{u} \cdot \nabla \tilde{v} + \left|\nabla \tilde{u}\right|^{p-2} \log \left(e+\left|\nabla \tilde{u}\right|\right) \nabla \tilde{u} \cdot \nabla \tilde{v} \right] dx\right. \notag \\
& \qquad \qquad + \left. \int_{\Omega} \frac{\left|\nabla \tilde{u}\right|^{p-1}}{e+|\nabla \tilde{u}|} \nabla \tilde{u} \cdot \nabla \tilde{v} dx \right| \notag \\
& \le (p+1) \left[\int_{\Omega} \left|\nabla \tilde{u}\right|^{p-1} \left|\nabla \tilde{v}\right| dx + \int_{\Omega} \left|\nabla \tilde{u}\right|^{p-1} \log \left(e+\left|\nabla \tilde{u}\right|\right) \left|\nabla \tilde{v}\right| dx\right].
\end{align}
Young's inequality gives us
\begin{align}\label{est-12}
\int_{\Omega} \left|\nabla \tilde{u}\right|^{p} \left|\nabla \tilde{v}\right| dx \le C \int_{\Omega} \left|\nabla \tilde{u}\right|^{p} dx + C \int_{\Omega} \left|\nabla \tilde{v}\right|^{p} dx.
\end{align}
Thanks to a modified form of Young's inequality in~\eqref{Young-ineq}, there exists a constant $C>0$ such that
\begin{align}\label{est-13}
& \int_{\Omega} \left|\nabla \tilde{u}\right|^{p-1} \log \left(e+\left|\nabla \tilde{u}\right|\right) \left|\nabla \tilde{v}\right| dx \le C\int_{\Omega} \left|\nabla \tilde{u}\right|^{p} \log \left(e+\left|\nabla \tilde{u}\right|\right) dx \notag\\
& \qquad \qquad + C\int_{\Omega} \left|\nabla \tilde{v}\right|^{p} \log \left(e+\left|\nabla \tilde{v}\right|\right) dx.
\end{align}
Substituting~\eqref{est-13} and~\eqref{est-12} into~\eqref{est-11}, one obtains that
\begin{align}\label{est-14}
\left|\left\langle \mathcal{L}\left(\tilde{u}\right), \tilde{v} \right\rangle\right| & \le C \int_{\Omega} \left|\nabla \tilde{u}\right|^{p} + \left|\nabla \tilde{u}\right|^{p} \log \left(e+\left|\nabla \tilde{u}\right|\right) dx \notag\\
& \qquad \qquad + C \int_{\Omega} \left|\nabla \tilde{v}\right|^{p} + \left|\nabla \tilde{v}\right|^{p} \log \left(e+\left|\nabla \tilde{v}\right|\right) dx \notag \\
& \le C \left(\|\tilde{u}\|_{\mathbb{W}_0} + \|\tilde{v}\|_{\mathbb{W}_0}\right) = C.
\end{align}
Due to~\eqref{est-14}, it allows us to conclude that 
\begin{align*}
\|\mathcal{L}(u)\|_{\mathbb{W}_0^*} = \sup_{\|v\|_{\mathbb{W}_0}\le 1} |\langle \mathcal{L}(u),v\rangle| \le C \|u\|_{\mathbb{W}_0}.
\end{align*}

The fact that the operator $\mathcal{L}$ is monotone comes from the ellipticity condition of the vector field $\mathcal{A}$. Indeed, let us define two auxiliary vector fields $V_p, V_{p\log}: \mathbb{R}^n \to \mathbb{R}^n$ by
\begin{align}\notag
V_p(\xi) := |\xi|^{\frac{p-2}{2}}\xi, \ V_{p\log}(\xi) := \left(p|\xi|^{p-2}\log (e+|\xi|) + \frac{|\xi|^{p-1}}{e+|\xi|}\right)^{\frac{1}{2}}\xi,
\end{align}
whenever $\xi \in \mathbb{R}^n$. These operators will be very useful and used mainly in our proofs later. With their aid, there exists a constant $C>0$ such that
\begin{align}\label{cond:A2}
\left( \mathcal{A}(\xi_1) - \mathcal{A}(\xi_2) \right) \cdot \left(\xi_1 - \xi_2 \right) \ge C \left[|V_p(\xi_1) - V_p(\xi_2)|^2 + |V_{p\log}(\xi_1) - V_{p\log}(\xi_2)|^2\right], 
\end{align}
for every $\xi_1, \xi_2 \in \mathbb{R}^n$. We refer the reader to \cite[Section 3.2]{BCM16} for details concerning the proof of~\eqref{cond:A2}. Then, it is readily verified that
\begin{align*}
\langle \mathcal{L}(u_1) - \mathcal{L}(u_2), u_1 - u_2 \rangle & = \int_{\Omega} \left( \mathcal{A}(\nabla u_1) - \mathcal{A}(\nabla u_2) \right) \cdot \left(\nabla u_1 - \nabla u_2 \right) dx \\
& \ge C \int_{\Omega} \left[|V_p(\nabla u_1) - V_p(\nabla u_2)|^2 + |V_{p\log}(\nabla u_1) - V_{p\log}(\nabla u_2)|^2\right],
\end{align*}
which implies that $\mathcal{L}$ is monotone.\\

In a next step, we prove that $\mathcal{L}$ satisfies the $(S)_+$ condition. Assume that $u_k \rightharpoonup u$ in $\mathbb{W}_0$ and $\limsup_{k \to \infty} \langle \mathcal{L}(u_k), u_k - u \rangle \le 0$, we need to show that $u_k \to u$ in $\mathbb{W}_0$. Firstly, it is obvious to see that
\begin{align*}
\lim_{k \to \infty} \langle \mathcal{L}(u), u_k - u \rangle = 0,
\end{align*}
which implies to
$$\limsup_{k \to \infty} \langle \mathcal{L}(u_k) - \mathcal{L}(u), u_k - u \rangle \le 0.$$
Combining this with the fact that $\mathcal{L}$ is monotone, it allows us to conclude that 
$$\lim_{n\to \infty}{\langle \mathcal{L}(u_k)-\mathcal{L}(u),u_k-u \rangle} = 0.$$ 
In other words, we have
\begin{align}\label{est-3-0}
& \lim_{k\to \infty} \fint_{\Omega} \left( \mathcal{A}(\nabla u_k) - \mathcal{A}(\nabla u) \right) \cdot \left(\nabla u_k - \nabla u \right)dx = 0. 
\end{align} 
It follows from~\cite[Section 3.2]{BCM16} that there exists a constant $C>0$ such that
\begin{align}\label{est-3-1}
& \fint_{\Omega}\left( \mathcal{A}(\nabla u_k) - \mathcal{A}(\nabla u) \right) \cdot \left(\nabla u_k - \nabla u \right)dx \ge C \fint_{\Omega} |\nabla u_k - \nabla u|^2 (|\nabla u_k|^2+|\nabla u|^2)^{\frac{p-2}{2}}dx \notag \\ 
& \qquad + C \fint_{\Omega} |\nabla u_k - \nabla u|^2 (|\nabla u_k|^2+|\nabla u|^2)^{\frac{p-2}{2}}\log (e+|\nabla u_k|+|\nabla u|)dx.
\end{align}
On the other hand, for every $\varepsilon>0$, Young's inequality leads to exist a constant $C_{\varepsilon} = C_{\varepsilon}(p,\varepsilon)>0$ such that
\begin{align}\label{est-3-2}
\fint_{\Omega} & |\nabla u_k - \nabla u|^p dx  \le \varepsilon \fint_{\Omega} |\nabla u|^pdx \notag \\
& \qquad + C_{\varepsilon} \fint_{\Omega} |\nabla u_k - \nabla u|^2 (|\nabla u_k|^2+|\nabla u|^2)^{\frac{p-2}{2}}dx.
\end{align}
Here, for the detailed proof of comparison estimate~\eqref{est-3-2}, we propose to the reader our previous work~\cite[Lemma 3.2]{TNPD22}. The same technique applied, it yields
\begin{align}
& \fint_{\Omega} |\nabla u_k - \nabla u|^p \log (e+|\nabla u_k - \nabla u|) dx \notag \\
& \qquad \le \varepsilon \left(\fint_{\Omega} |\nabla u|^p \log (e+|\nabla u|)dx + \fint_{\Omega} |\nabla u_k|^p \log (e+|\nabla u_k|)dx\right) \notag \\
& \qquad \quad + C_{\varepsilon} \fint_{\Omega} |\nabla u_k - \nabla u|^2 (|\nabla u_k|^2+|\nabla u|^2)^{\frac{p-2}{2}} \log (e+|\nabla u_k| + |\nabla u|)dx. \label{est-3-3a}
\end{align}
Indeed, to obtain~\eqref{est-3-3a}, it is worth noting that the following inequality
\begin{align}
|\nabla u_k - \nabla u|^p & \log (e+|\nabla u_k - \nabla u|) \notag \\
& \le C|\nabla u_k - \nabla u|^2 (|\nabla u_k|^2+|\nabla u|^2)^{\frac{p-2}{2}} \log (e+|\nabla u_k| + |\nabla u|) \notag
\end{align}
holds whenever $p \ge 2$. When $1 < p <2$ and for every $\varepsilon>0$, thanks to Young's inequality, it gives
\begin{align}
& |\nabla u_k - \nabla u|^p \log (e+|\nabla u_k - \nabla u|) \notag \\
& = (|\nabla u_k|^2+|\nabla u|^2)^{\frac{p(2-p)}{4}} \left( (|\nabla u_k|^2+|\nabla u|^2)^{\frac{p-2}{2}} |\nabla u_k - \nabla u|^2\right)^{\frac{p}{2}} \log (e+|\nabla u_k-\nabla u|) \notag\\
& \le \varepsilon (|\nabla u_k|+|\nabla u|)^p \log (e+|\nabla u_k|+|\nabla u|) \notag \\
& \qquad \qquad + C_{\varepsilon} |\nabla u_k - \nabla u|^2 (|\nabla u_k|^2+|\nabla u|^2)^{\frac{p-2}{2}} \log (e+|\nabla u_k| + |\nabla u|) \notag \\
& \le C \varepsilon \left(|\nabla u_k|^p \log (e+|\nabla u_k|) + |\nabla u|^p \log (e+|\nabla u|) \right) \notag \\
& \qquad \qquad + C_{\varepsilon} |\nabla u_k - \nabla u|^2 (|\nabla u_k|^2+|\nabla u|^2)^{\frac{p-2}{2}} \log (e+|\nabla u_k| + |\nabla u|), \notag
\end{align}
which leads to~\eqref{est-3-3a}. Moreover, by using the following inequality
\begin{align*}
|\nabla u_k|^p \log (e+|\nabla u_k|) \le C \big[|\nabla u|^p \log (e+|\nabla u|) + |\nabla u - \nabla u_k|^p \log (e+|\nabla u - \nabla u_k|)\big],
\end{align*}
we can deduce from~\eqref{est-3-3a} to
\begin{align}\label{est-3-3}
& \fint_{\Omega} |\nabla u_k - \nabla u|^p \log (e+|\nabla u_k - \nabla u|) dx \le \varepsilon \fint_{\Omega} |\nabla u|^p \log (e+|\nabla u|)dx\notag  \\
& \qquad  + C_{\varepsilon} \fint_{\Omega} |\nabla u_k - \nabla u|^2 (|\nabla u_k|^2 +|\nabla u|^2)^{\frac{p-2}{2}} \log (e+|\nabla u_k| + |\nabla u|)dx.
\end{align}
Combining all estimates in~\eqref{est-3-1},~\eqref{est-3-2} and~\eqref{est-3-3}, one obtains that
\begin{align}\label{est-3-4}
 \fint_{\Omega} & |\nabla u_k - \nabla u|^p  + |\nabla u_k - \nabla u|^p \log (e+|\nabla u_k - \nabla u|) dx \notag \\
& \qquad \le \varepsilon  \fint_{\Omega} |\nabla u|^p + |\nabla u|^p \log (e+|\nabla u|) dx \notag \\
& \qquad \qquad + C_{\varepsilon}\fint_{\Omega}\left( \mathcal{A}(\nabla u_k) - \mathcal{A}(\nabla u) \right) \cdot \left(\nabla u_k - \nabla u \right)dx,
\end{align}
for all $\varepsilon>0$. For every $\delta \in (0,1)$, let us choose $\varepsilon>0$ in~\eqref{est-3-4} such that
\begin{align*}
\varepsilon \left( 1+ \fint_{\Omega} |\nabla u|^p + |\nabla u|^p \log (e+|\nabla u|) dx\right) < \delta^p.
\end{align*}
By~\eqref{est-3-0}, there exists $k_0 \in \mathbb{N}$ such that
\begin{align*}
\fint_{\Omega}\left( \mathcal{A}(\nabla u_k) - \mathcal{A}(\nabla u) \right) \cdot \left(\nabla u_k - \nabla u \right)dx \le \varepsilon C_{\varepsilon}^{-1}, \quad \forall k \ge k_0.
\end{align*}
Therefore, it follows from~\eqref{est-3-4} that
\begin{align}
 \fint_{\Omega} |\nabla u_k - \nabla u|^p & + |\nabla u_k - \nabla u|^p \log (e+|\nabla u_k - \nabla u|) dx \notag \\
& \qquad \le \varepsilon \left( 1+ \fint_{\Omega} |\nabla u|^p + |\nabla u|^p \log (e+|\nabla u|) dx\right) < \delta^p, \label{est-3-5}
\end{align}
for all $k \ge k_0$. In what follows, when no confusion arises, we shall always consider $k \ge k_0$. From~\eqref{est-3-5}, we have
\begin{align*}
0< \|\nabla u_k - \nabla u\|_p < \delta < 1,
\end{align*}
and at this stage, apply~\eqref{tge1} with $\alpha=1$, we arrive at
\begin{align}
\fint_{\Omega} |\nabla u_k - \nabla u|^p & \log \left(e+\frac{|\nabla u_k - \nabla u|}{\|\nabla u_k - \nabla u\|_p}\right) dx \notag \\
& \le \fint_{\Omega} |\nabla u_k - \nabla u|^p \log (e+|\nabla u_k - \nabla u|) dx \notag \\
& \qquad + \fint_{\Omega} |\nabla u_k - \nabla u|^p \log \left(\frac{1}{\|\nabla u_k - \nabla u\|_p}\right) dx. \label{est-3-6}
\end{align}
Further,~\eqref{tge1-2} with specific $\alpha=p=1$, we readily obtain
\begin{align*}
\fint_{\Omega} |\nabla u_k - \nabla u|^p & \log \left(\frac{1}{\|\nabla u_k - \nabla u\|_p}\right) dx \\
& \le \frac{1}{e} \fint_{\Omega} |\nabla u_k - \nabla u|^p \frac{1}{\|\nabla u_k - \nabla u\|_p} dx  \le  \frac{1}{e} \delta^{p-1},
\end{align*}
and due to~\eqref{est-3-6}, it gives
\begin{align}
\fint_{\Omega} |\nabla u_k - \nabla u|^p & \log \left(e+\frac{|\nabla u_k - \nabla u|}{\|\nabla u_k - \nabla u\|_p}\right) dx \notag \\
& \le \fint_{\Omega} |\nabla u_k - \nabla u|^p \log (e+|\nabla u_k - \nabla u|) dx + \frac{1}{e} \delta^{p-1}. \label{est-3-7}
\end{align}
Moreover, from~\eqref{norm-W0}, the norm in $\mathbb{W}_0$ and in virtue of inequality~\eqref{est:Lbb} of Lemma~\ref{lem:Lbb}, we have
\begin{align}
\|u_k - u\|_{\mathbb{W}_0} & = \|\nabla u_k - \nabla u\|_{\mathbb{L}^{p\log}(\Omega)} \notag \\
& \le \left[\fint_{\Omega} |\nabla u_k - \nabla u|^p  + |\nabla u_k - \nabla u|^p \log \left(e+\frac{|\nabla u_k - \nabla u|}{\|\nabla u_k - \nabla u\|_p}\right) dx\right]^{\frac{1}{p}}. \label{est-3-8}
\end{align}
Substituting~\eqref{est-3-5} and~\eqref{est-3-7} into~\eqref{est-3-8}, we infer that
\begin{align}\notag
\|u_k - u\|_{\mathbb{W}_0} & \le \left[\delta^p + \frac{1}{e} \delta^{p-1}\right]^{\frac{1}{p}}.
\end{align}
Therefore, it allows us to conclude that 
\begin{align}\notag
& \lim_{k \to \infty} \|u_k - u\|_{\mathbb{W}_0} = 0,
\end{align}
or $u_k \to u$ in $\mathbb{W}_0$. The proof is complete.
\end{proof}

Once having the previous Lemma at hand, we are now in a position to accomplish the proof of Theorem~\ref{theo-main}, the main existence result in this study.\\

\begin{proof}[Proof of Theorem~\ref{theo-main}]
Let $I^*: (L^{q}\log L(\Omega))^* \to \mathbb{W}_0^*$ be the adjoined operator of the compact embedding $I: \mathbb{W}_0 \to L^{q}\log L(\Omega)$. Assume that 
$$\tilde{\mathcal{N}}_F: \mathbb{W}_0 \subset L^{q}\log L(\Omega) \to (L^{q}\log L(\Omega))^*$$ is the Nemytskij operator associated to $F$ and set $\mathcal{N}_F = I^* \circ \tilde{\mathcal{N}}_F$. According to the assumption~\eqref{Asump-f-1}$_1$, we arrive at
\begin{align}\notag
|F(x,u,\nabla u)| & \le g(x) + \mu_1 |u|^{q-1}\log (e+|u|) + \mu_2 |\nabla u|^{\frac{p}{q'}}\log \left(e+|\nabla u|^{\frac{q}{p}}\right). 
\end{align}
Thanks to inequality~\eqref{G-star} in Lemma~\ref{lem-Young}, it follows that if the Young function $G \in \Delta_2 \cap \nabla_2$ and $u \in L^{G}(\Omega)$, then $\frac{G(|u|)}{|u|} \in (L^{G}(\Omega))^* = L^{G^*}(\Omega)$. Combining this with the following facts
\begin{align*}
u \in L^{q}\log L(\Omega) = L^{\varphi_{q}}(\Omega), \mbox{ and } |\nabla u|^{\frac{q}{p}} \in L^{q}\log L(\Omega) \ \mbox{ (Remark~\ref{rmk-1})},
\end{align*}
it leads to
\begin{align*}
\frac{\varphi_q(|u|)}{|u|} \in (L^{q}\log L(\Omega))^*, \mbox{ and } \frac{\varphi_{q}(|\nabla u|^{\frac{q}{p}})}{|\nabla u|^{\frac{q}{p}}} \in (L^{q}\log L(\Omega))^*, 
\end{align*}
From this reasoning, we conclude that the operator $\mathcal{N}_F$ maps $\mathbb{W}_0$ into $\mathbb{W}_0^*$ by Lemma~\ref{lem-Hold}. Next, let us introduce a new operator $\mathcal{G}: \mathbb{W}_0 \to \mathbb{W}_0^*$ defined by
\begin{align}\label{def-G-u}
\mathcal{G}(u) = \mathcal{L}(u) - \mathcal{N}_F(u), \quad u \in \mathbb{W}_0,
\end{align}
and it is clear that this operator maps bounded sets into bounded sets. 

At this stage, let us now show that $\mathcal{G}$ is pseudo-monotone, that means
\begin{align}\label{pseudo-1}
u_k \rightharpoonup u \mbox{ in } \mathbb{W}_0 \ \mbox{ and } \ \limsup_{k \to \infty} \langle \mathcal{G}(u_k), u_k - u \rangle \le 0,
\end{align}
imply $u_k \to u$ in $\mathbb{W}_0$. Assume that the sequence $(u_k)$ satisfies~\eqref{pseudo-1}, one need to show that $u_k \to u$ in $\mathbb{W}_0$. With Lemma~\ref{lem:embed} at hand, the embedding $\mathbb{W}_0 \hookrightarrow L^{q}\log L(\Omega)$ is compact. It then leads to $u_k \to u$ in $L^{q}\log L(\Omega)$. On the other hand, by assumption~\eqref{Asump-f-1}$_1$ and H\"older's inequality~\eqref{Hold-ineq} in Lemma~\ref{lem-Hold}, we infer that
\begin{align}\notag
\left|\int_{\Omega}F(x,u_k,\nabla u_k)(u_k-u) dx \right|  & \le \int_{\Omega} g(x) |u_k-u| dx \\
& \quad + \mu_1 \int_{\Omega} |u_k|^{q-1}\log (e+|u_k|) |u_k-u| dx \notag\\
& \qquad  + \mu_2 \int_{\Omega} |\nabla u_k|^{\frac{p}{q'}}\log \left(e+|\nabla u_k|^{\frac{q}{p}}\right) |u_k-u| dx \notag \\
& \le T_k \|u_k-u\|_{L^{q}\log L(\Omega)}, \label{est-F1}
\end{align}
where the bounded sequence $T_k$ is given by
\begin{align*}
T_k = \|g\|_{(L^{q}\log L(\Omega))^*} + \left\|\frac{\varphi_q(u_k)}{|u_k|}\right\|_{(L^{q}\log L(\Omega))^*} + \left\|\frac{\varphi_{q}(|\nabla u_k|^{p/q})}{|\nabla u_k|^{p/q}}\right\|_{(L^{q}\log L(\Omega))^*}.
\end{align*}
Passing to the limit $k \to \infty$ in~\eqref{est-F1}, it gives
\begin{align*}
\lim_{k \to \infty} \int_{\Omega}F(x,u_k,\nabla u_k)(u_k-u) dx = 0,
\end{align*}
and taking~\eqref{pseudo-1} into account, it ensures that
\begin{align}\notag
\limsup_{k \to \infty} \langle \mathcal{A}(u_k), u_k - u\rangle = \limsup_{k \to \infty} \langle \mathcal{G}(u_k), u_k - u \rangle \le 0.
\end{align} 
Invoking Lemma~\ref{lem:properties-L}, we are allowed to conclude that $u_k \to u$ in $\mathbb{W}_0$. Hence, one gets that $\mathcal{G}(u_k) \to \mathcal{G}(u)$ in $\mathbb{W}_0^*$ by the continuity of $\mathcal{G}$ and therefore, $\mathcal{G}$ is pseudo-monotone.

The last step is devoted to showing that $\mathcal{G}$ is coercive, which means
\begin{align}\label{coer-G}
\lim_{\|u\|_{\mathbb{W}_0} \to \infty}\frac{\langle \mathcal{G}(u), u \rangle}{\|u\|_{\mathbb{W}_0}} = \infty.
\end{align}
For every $u \in \mathbb{W}_0$, let us present $\langle \mathcal{G}(u), u \rangle$ as follows
\begin{align}\label{coercive-0}
\langle \mathcal{G}(u), u \rangle & = |\Omega| \left(p\fint_{\Omega} |\nabla u|^p + |\nabla u|^p \log(e + |\nabla u|) dx + \fint_{\Omega} \frac{|\nabla u|^{p+1}}{e+|\nabla u|}dx \right. \notag \\
& \qquad \qquad \left. - \fint_{\Omega} F(x,u,\nabla u) u dx\right).
\end{align}
Combining assumption~\eqref{Asump-f-1}$_2$ with inequality~\eqref{eig-value}, it gives us that
\begin{align}\label{coercive-1}
\fint_{\Omega} F(x,u,\nabla u) u dx &\le \fint_{\Omega} h(x) dx + \mu_3 \fint_{\Omega} |u|^p  dx  + \mu_4 \fint_{\Omega} |\nabla u|^p \log(e + |\nabla u|) dx \notag \\
 & \le \|h\|_1 +  \frac{\mu_3}{\lambda_{1,p}} \fint_{\Omega} |\nabla u|^p  dx + \mu_4 \fint_{\Omega} |\nabla u|^p \log(e + |\nabla u|) dx.
\end{align}
Then, substituting~\eqref{coercive-1} into~\eqref{coercive-0}, it leads to
\begin{align}\label{coercive-2}
\langle \mathcal{G}(u), u \rangle & \ge |\Omega| \left[\left(p - \frac{\mu_3}{\lambda_{1,p}}\right) \fint_{\Omega} |\nabla u|^p  dx  + (p-\mu_4)\fint_{\Omega}|\nabla u|^p \log(e + |\nabla u|) dx - \|h\|_{1}\right].
\end{align}
On the other hand, it is obviously that
\begin{align*}
\fint_{\Omega}|\nabla u|^p \log \left(e + \frac{|\nabla u|}{\|\nabla u\|_p}\right) dx \le \fint_{\Omega}|\nabla u|^p \log(e + |\nabla u|) dx
\end{align*}
whenever $\|\nabla u\|_p \ge 1$. In the other case $0< \|\nabla u\|_p < 1$, thanks to~\eqref{tge1} and~\eqref{tge1-2}, there holds
\begin{align}\label{coercive-3}
\fint_{\Omega}|\nabla u|^p \log \left(e + \frac{|\nabla u|}{\|\nabla u\|_p}\right) dx & \le \fint_{\Omega}|\nabla u|^p \log(e + |\nabla u|) dx + \fint_{\Omega}|\nabla u|^p \log(\|\nabla u\|_p^{-1}) dx \notag \\
& \le \fint_{\Omega}|\nabla u|^p \log(e + |\nabla u|) dx + \frac{1}{ep}.
\end{align}
For this reason, inequality~\eqref{coercive-3} holds whenever $\|\nabla u\|_p > 0$. Combining~\eqref{coercive-2} and~\eqref{coercive-3}, one obtains that
\begin{align}\label{coercive-4}
\langle \mathcal{G}(u), u \rangle  & \ge |\Omega| \left[\left(p - \frac{\mu_3}{\lambda_{1,p}}\right) \fint_{\Omega} |\nabla u|^p  dx \right. \notag \\ & \qquad \left.  + (p-\mu_4)\fint_{\Omega}|\nabla u|^p \log \left(e + \frac{|\nabla u|}{\|\nabla u\|_p}\right) dx - \frac{p-\mu_4}{ep} - \|h\|_{1}\right] \notag \\
& \ge |\Omega| \left[ \mu_0 [\nabla u]_{\mathbb{L}^{p\log}(\Omega)}^p - \frac{p-\mu_4}{ep} - \|h\|_{1}\right] \notag \\
& \ge |\Omega| \left[ \mu_0 \|u\|_{\mathbb{W}_0}^p - \frac{p-\mu_4}{ep} - \|h\|_{1}\right],
\end{align}
where $\mu_0$ is given by 
$$\mu_0 = \min\left\{p - \frac{\mu_3}{\lambda_{1,p}}; p-\mu_4\right\}.$$ 
Assumption~\eqref{cond-mu} guarantees that the constant $\mu_0$ is positive. Since $p>1$ and $\mu_0>0$, we may conclude that~\eqref{coer-G} holds from~\eqref{coercive-4}. This mean the operator $\mathcal{G}$ is coercive.

At this point, as our previous proofs, all hypotheses of Theorem~\ref{theo-CLM} hold true for $W = \mathbb{W}_0$ and operator $\mathcal{G}$ defined in~\eqref{def-G-u}, hence there exists at least one $u \in \mathbb{W}_0$ such that $\mathcal{G}(u) = 0$. In conclusion, the problem~\eqref{main-eq} admits at least one weak solution in $\mathbb{W}_0$. The proof is now complete.
\end{proof}



\begin{thebibliography}{99}
\footnotesize

\bibitem{AM01} E. Acerbi, G. Mingione, {\em Regularity results for a class of functionals with nonstandard growth}, Arch. Rational Mech. Anal. {\bf 156} (2001), 121--140.

\bibitem{AM05} E. Acerbi, G. Mingione, {\em Gradient estimates for the $p(x)$-Laplacean system}, J. Reine Angew. Math. {\bf 584} (2005), 117--148.





\bibitem{BCM15} P. Baroni, M. Colombo, G. Mingione, {\em Harnack inequalities for double phase functionals}, Nonlinear Anal. {\bf 121} (2015), 206--222.

\bibitem{BCM16} P. Baroni, M. Colombo, G. Mingione, {\em Non-autonomous functionals, borderline cases and related function classes}, St. Petersburg Math. J. {\bf 27} (2016), 347--379.

\bibitem{BCM18} P. Baroni, M. Colombo, G. Mingione, {\em Regularity for general functionals with double phase}, Calc. Var. Partial Differ. Equ. {\bf 57}(2) (2018) 62, 48 pp.

\bibitem{BM2020}  L. Beck, G. Mingione, {\em Lipschitz bounds and nonuniform ellipticity}, Comm. Pure Appl. Math. {\bf 73}(5) (2020), 944--1034.

\bibitem{BS1988} C. Bennett, R. Sharpley, {\em Interpolation of Operators}, Academic Press, 1988.

\bibitem{Brezis} H. Brezis, {\em Equations et in\'equations non lin\'eaires dans les espaces vectoriels \'endualit\'e}, Universit\'e de Grenoble. Annales de l’Institut Fourier {\bf 18}, (1968), 115--175.

\bibitem{Browder} F. E. Browder, {\em Nonlinear elliptic boundary value problems}, Bull. Amer.
Math. Soc. {\bf 69} (1963), 862--874.


\bibitem{BO2017} S.-S. Byun, J. Oh, {\em Global gradient estimates for non-uniformly elliptic equations}, Calc. Var. Partial Differ. Equ. {\bf 56} (2017), 46.

\bibitem{CLM07} S. Carl, V. K. Le, D. Motreanu, {\em Nonsmooth variational problems and their inequalities}, Springer, New York, 2007.

\bibitem{C96} A. Cianchi, {\em A sharp embedding theorem for Orlicz–Sobolev spaces}, Indiana Univ. Math. J. {\bf 45} (1996),  39--65.


\bibitem{CM2015} M. Colombo, G. Mingione, {\em Bounded minimisers of double phase variational integrals}, Arch. Ration. Mech. Anal. {\bf 218}(1) (2015), 219--273.

\bibitem{CM2015_2} M. Colombo, G. Mingione, {\em Regularity for double phase variational problems}, Arch. Ration. Mech. Anal. {\bf 215}(2) (2015), 443--496.


\bibitem{FM2019} C. De Filippis, G. Mingione, {\em A borderline case of Calder\'on-Zygmund estimates for non-uniformly elliptic problems}, St. Petersburg Math. J. {\bf 31}(2020) 455-477.


\bibitem{FM2020} C. De Filippis, G. Mingione, {\em Manifold constrained non-uniformly elliptic problems}, J. Geom. Anal. {\bf 30} (2020), 1661--1723.

\bibitem{FM2020_2} C. De Filippis, G. Mingione, {\em On the regularity of minima of non-autonomous functionals}, J. Geom. Anal. {\bf 30}(2) (2020), 1584--1626.

\bibitem{FM2023} C. De Filippis, G. Mingione, {\em  Regularity for double phase problems at nearly linear growth}, Arch. Rational Mech. Anal. {\bf 247}, 85 (2023).


\bibitem{EMW2022} S. El Manouni, G. Marino, P. Winkert, {\em Existence results for double phase problems depending on Robin and Steklov eigenvalues for the $p$-Laplacian}, Adv. Nonlinear Anal. {\bf 11} (2022), 304--320.

%
%
%
%
%
%
%

\bibitem{GW19} L. Gasi\'nski, P. Winkert, {\em Existence and uniqueness results for double phase problems with convection term}, J. Differential Equations {\bf 268}(8) (2020), 4183--4193.

\bibitem{GP1987} J. P. Garc\'ia Azorero, I. Peral Alonso, {\em Existence and nonuniqueness for the $p$-Laplacian: nonlinear eigenvalues}, Comm. Partial Differential Equations {\bf 12}(12) (1987), 1389--1430. 

\bibitem{HS1966} P. Hartman, G. Stampacchia, {\em On some non-linear elliptic differential functional equations}, Acta Mathematica {\bf 115}(1) (1966), 271--310.

\bibitem{HH19} P. Harjulehto, P. H\"ast\"o, {\em Orlicz Spaces and Generalized Orlicz Spaces},  Lecture Notes in Mathematics,  Springer, 1st ed. 2019.

\bibitem{IV99} T. Iwaniec, A. Verde, {\em On the Operator $L(f)=f\log |f|$},  J. Funct. Anal. {\bf 169} (1999), 391--420.

\bibitem{LM2019} J. Lang, O. Mendez, {\em Analysis on Function Spaces of Musielak-Orlicz Type}, Chapman \& Hall/CRC Monographs and Research Notes in Mathematics, 2019.

\bibitem{L06} A. Le, {\em Eigenvalue problems for the $p$-Laplacian}, Nonlinear Anal. {\bf 64}(5) (2006), 1057--1099.

\bibitem{Lin90} P. Lindqvist, {\em On the equation $\mathrm{div}(|\nabla u|^{p-2}\nabla u) +\lambda|u^{p-2}|u=0$}, Proc. Amer. Math. Soc. {\bf 109}(1) (1990), 157--164.


\bibitem{LD18} W. Liu, G. Dai, {\em Existence and multiplicity results for double phase problem}, J. Differ. Equ. {\bf 265}(9) (2018), 4311--4334.

\bibitem{Marcellini1989} P. Marcellini, {\em  Regularity of minimisers of integrals of the calculus of variations with non-standard growth conditions}, Arch. Rat. Mech. Anal. {\bf 105} (1989), 267--284.

\bibitem{Marcellini1991} P. Marcellini, {\em  Regularity and existence of solutions of elliptic equations with $p,q$-growth conditions},  J. Diff. Equ. {\bf 90} (1991), 1--30.

\bibitem{MW2019} S. A. Marano, P. Winkert, {\em On a quasilinear elliptic problem with convection term and nonlinear boundary condition}, Nonlinear Anal. {\bf 187} (2019), 159--169.


\bibitem{MS1999} G. Mingione, F. Siepe, {\em Full $C^{1,\alpha}$-regularity for minimizers of integral functionals with $L \log L$-growth},  Z. Anal. Anw. {\bf 18} (1999), 1083--1100.

\bibitem{MR2021}  G. Mingione, V. D. R\c{a}dulescu, {\em Recent developments in problems with nonstandard growth and nonuniform ellipticity}, J. Math. Anal. Appl. {\bf 501}(1) (2021),  125197, 41 pp.


\bibitem{Minty} G. Minty, {\em On a monotonicity method for the solution of non-linear equations in Banach spaces}, Proc. Nat. Acad. Sci. U.S.A. {\bf 50} (1963), 1038--1041.



\bibitem{PS18} K. Perera, M. Squassina, {\em Existence results for double-phase problems via Morse theory}, Commun. Contemp. Math. {\bf 20}(2) (2018), 1750023, 14 pp.

\bibitem{R2019} V. D. R\v{a}dulescu, {\em Isotropic and anisotropic double-phase problems: old and new}, Opuscula Math. {\bf 39}(2) (2019), 259--279.



\bibitem{Stein} E. M. Stein, {\em Note on the Class $L\mathrm{log}L$}, Studia Math. {\bf 32} (1969), 305--310.

\bibitem{TN21} M.-P. Tran, T.-N. Nguyen, {\em Global Lorentz estimates for non-uniformly nonlinear elliptic equations via fractional maximal operators}, J. Math. Anal. Appl. {\bf 501}(1) (2021), 124084.

\bibitem{TNPD22} M.-P. Tran, T.-N. Nguyen, L.-T.-N. Pham, T.-T.-T. Dang, {\em Weighted Lorentz estimates for non-uniformly elliptic problems with variable exponents}, Manuscripta Math. {\bf 172} (2023), 1227--1244.

\bibitem{VW2022} F. Vetro, P. Winkert, {\em Existence, uniqueness and asymptotic behavior of parametric anisotropic $(p,q)$-equations with convection}, Appl. Math. Optim. {\bf 86}(18) (2022), 18 pp.


\bibitem{ZR2018} Q. Zhang, V. D. R\v{a}dulescu, {\em Double phase anisotropic variational problems and combined effects of reaction and absorption terms}, J. Math. Pures Appl. (9) {\bf 118} (2018), 159--203.

\bibitem{Zhikov1986}  V. V. Zhikov, {\em Averaging of functionals of the calculus of variations and elasticity theory}, Izv. Akad. Nauk SSSR Ser. Mat. {\bf 50}(4) (1986), 675--710.

%
%

\bibitem{Zhikov2011} V. V. Zhikov, {\em On variational problems and nonlinear elliptic equations with nonstandard growth conditions}, J. Math. Sci. {\bf 173} (2011), 463--570.

\end{thebibliography}
\end{document}